\documentclass[final]{siamltex}

\usepackage{epsfig}
\usepackage{amsmath,amsfonts,amssymb}
\newcommand{\smat}[1]{ \left[\begin{smallmatrix} #1 \end{smallmatrix}\right]}

\def\setC{\,\mathbb{C}}
\def\setR{\mathbb{R}}

\def\setZ{\mathbb{Z}}
\def\setP{\mathbb{P}}

\newtheorem{rumatemppu}{Example}
\newenvironment{example}{\begin{rumatemppu}\rm}{\end{rumatemppu}}
\newtheorem{rumatemppuu}{Algorithm}

\title{THE PRODUCT OF MATRIX SUBSPACES}
\author{
Marko Huhtanen\thanks{
Department of Mathematics and Systems Analysis,
Aalto University, 
Box 1100
FIN-02015,
Finland,
({\tt Marko.Huhtanen@hut.fi}). Supported by the Academy of Finland.}
}

\begin{document}
\maketitle
\begin{abstract} 
In factoring matrices into the product of two matrices  
operations are  typically performed with elements restricted to matrix subspaces.
Such  modest structural assumptions are realistic, for example, 
in large scale computations.
This paper is concerned with analyzing associated matrix
geometries. 
Curvature of the product of two matrix subspaces
is assessed.
As an analogue of the internal Zappa-Sz\'ep product 
of a group, 
the  notion of factorizable matrix subspace arises.
Interpreted in this way, several classical instances
are encompassed by this  structure.
The Craig-Sakamoto theorem fits naturally into this framework.

\end{abstract}


\pagestyle{myheadings}
\thispagestyle{plain}

\markboth{MARKO HUHTANEN   
}{PRODUCT OF MATRIX SUBSPACES} 

\section{Introduction}  
This paper is concerned with the set consisting of
matrix products of elements restricted to
 matrix subspaces 
$\mathcal{V}_1$ and $\mathcal{V}_2$ 
of $\setC^{n \times n}$ over $\setC$ (or $\setR$). 
Matrix subspaces appear regularly in 
large scale computational problems where only
modest structural assumptions can be made.\footnote{In large
scale problems, a matrix subspace typically defined by fixing 
a sparsity structure with $O(n)$ nonzero entries.}  
For example,
factorization problems \cite{WU,LA,PHA}
are typically  subspace problems.
The study of matrix subspaces can be classified as being 
finite dimensional operator space theory. For operator space
theory, see \cite{EFF,VPA}.
Regarding the geometry associated with matrix factoring \cite{HR,BHU},
the  set of products is defined as   
$$
\mathcal{V}_1\mathcal{V}_2=\{V_1V_2\,:\, V_1\in\mathcal{V}_1 \,
\mbox{ and }\, V_2\in\mathcal{V}_2 \}.\footnote{In algebraic geometry,
when $\mathcal{V}_1$
and $\mathcal{V}_2$ are treated as projective Hilbert spaces,
the tensor product is the more usual object  of interest  
(given by the Segre map).}
$$
As illustrated by the LU factorization, both complete and incomplete,  
as well as the singular value decomposition,
this structure is  ubiquitous. 
(The subset of $\setC^{n \times n}$ of
 matrices of rank $k$ at most is also the product
of two matrix subspaces of $\setC^{n \times n}.$) 
These examples also underscore  how different aims
and  geometries the set of products can have.

The set of products is not the most natural structure
from the computational view-point of factoring.\footnote{Being computationally far
more accessible, it is likely that the set
$\mathcal{V}_1\,{\rm Inv }(\mathcal{V}_2)=\{V_1V_2^{-1}\,:\, V_1\in
\mathcal{V}_1 \,
\mbox{ and }\, V_2\in \mathcal{V}_2 \cap {\rm GL}(n,\setC)\}$
is more important; see
\cite{HR,BHU}.}
It is encountered regularly, though.
Through the LU factorization, band matrices is one such instance, even though
it is exceptional by being a matrix subspace. 
In general, the set of products is constructible and certainly not flat.
Thereby one is led
to ask how curved $\mathcal{V}_1\mathcal{V}_2$ 
is. By locally inspecting the image of the smooth map 
\begin{equation}\label{matpro}
(V_1,V_2)\longmapsto V_1V_2,
\end{equation}
this can be assessed Riemannian geometrically.
This approach turns out to have a global character, leading to a 
necessary and sufficient condition for
the flatness of the whole image of \eqref{matpro}.

Vanishing curvature is intriguing  by the fact that
then the associated matrix subspace factors, i.e.,
it is not a so-called irreducible matrix subspace. 
And conversely, it is a fundamental problem to
recover whether a given matrix subspace is irreducible.
Once interpreted this way, it becomes clear that this notion has many 
appearances
already in the commutative case, such as integer factorization and polynomial factoring.
In the noncommutative case of $\dim \mathcal{V}_1=\dim \mathcal{V}_2=2$, 
an appropriate general treatment of the Craig-Sakamoto theorem in statistics is shown to correspond to such an instance.

The paper is organized as follows. In Section \ref{CS} the product of 
matrix subspaces is defined. 
Closedness is addressed
and some fundamental linear algebraic notions
are recalled. 
The structure is
illustrated with  several examples. 
Through the Craig-Sakamoto theorem,
special attention is paid to 
the three dimensional case. 
Bihomogeneous 
polynomial maps of bidegree $(1,1)$ are associated 
with the problem.
In Section \ref{GPMS} the notion of factorizable matrix subspace is introduced.
The curvature of the set of products is assessed locally. Based on this,
a necessary and  sufficient condition for
the flatness of the set of products is given.

\section{The product of two matrix subspaces}\label{CS}  
The product of two matrix subspaces is defined and exemplifed.
Special attention is paid to
a generalization of
the Craig-Sakamoto theorem. 
Bihomogeneous polynomial maps are associated with 
the product of two matrix subspaces.

\subsection{Matrix subspaces and the Craig-Sakamoto theorem}
Assume $\mathcal{V}$ is a matrix subspace of $\setC^{n \times n}$ 
over $\setC$ (or $\setR$). 
Depending on whether $\mathcal{V}$  contains invertible 
elements, the matrix subspace  
is called either nonsingular or singular \cite{HR}.
Generically, a matrix subspace is nonsingular \cite{HR2}. 
Then 
its subset consisting of invertible elements
is open and dense \cite{HR}.

Two matrix subspaces $\mathcal{V}$ and $\mathcal{W}$ are said to be 
equivalent
if 
$\mathcal{W}=X\mathcal{V}Y^{-1}$
holds  for invertible matrices
$X,Y\in \setC^{n \times n}$. In view of their properties, equivalent matrix
subspaces can in many ways be regarded as being indistinguishable.

\begin{definition}\label{prods}
Suppose $\mathcal{V}_1$ and $\mathcal{V}_2$ are   matrix subspaces 
of $\setC^{n \times n}$ over $\setC$ (or $\setR$). Then 
$$\mathcal{V}_1\mathcal{V}_2=\{V_1V_2\,:\, V_1\in\mathcal{V}_1 \,
\mbox{ and }\, V_2\in\mathcal{V}_2 \}
$$
is said to be the set of  products of  
$\mathcal{V}_1$ and
$\mathcal{V}_2$.
\end{definition}

Clearly,  we have a  homogeneous set,
i.e., there holds $t  \,\mathcal{V}_1\mathcal{V}_2=\mathcal{V}_1\mathcal{V}_2$
for any nonzero scalar $t$. 
However, unlike a matrix subspace, 
the set of products need not be closed.
(Consider, for example, the LU factorization.)
Closedness is certainly of importance, e.g., in 
stable numerical computations.\footnote{It is well-known
that the compution of an LU factorization is not a stable
process unless one uses partial pivoting.}

\begin{theorem}\label{sutu} 
Assume  $\mathcal{V}_1$ and
$\mathcal{V}_2$ are matrix subspaces of 
$\setC^{n \times n}$ over $\setC$ (or $\setR$) such that
$V_1V_2=0$ if and only if either $V_1=0$ or $V_2=0$.
Then $\mathcal{V}_1\mathcal{V}_2$ is closed.\footnote{A natural related problem is as follows.
Assume  a matrix subspace $\mathcal{V}_1$ of $\setC^{n \times n}$  over $\setC$ (or $\setR$) is given.
Find a matrix subspace $\mathcal{V}_2$ of $\setC^{n \times n}$  over $\setC$ (or $\setR$)  
of the largest possible dimension such that the assumptions of
Theorem \ref{sutu} are satisfied.}
\end{theorem}

\begin{proof} Consider the matrix product 
$$(V_1,V_2)\longmapsto V_1V_2$$
with $V_1\in \mathcal{V}_1$ and $V_2\in \mathcal{V}_2$.
For any scalars $s$ and $t$ we have $(sV_1,tV_2)\mapsto stV_1V_2$.
Thereby, because of the assumptions, 
we can regard the matrix product as a map from the product of projective spaces
$\mathbf{P}(\mathcal{V}_1) \times \mathbf{P}(\mathcal{V}_2)$ to 
the projective space $\mathbf{P}(\setC^{n \times n})$. This is a map from a
 compact space to a compact space.
Consequently, by the closed map lemma
\cite[Lemma 4.25]{leetm}, the image of 
$\mathbf{P}(\mathcal{V}_1) \times \mathbf{P}(\mathcal{V}_2)$ is closed.
\end{proof}

For the computation of 
$$\{(V_1,V_2)\in \mathcal{V}_1\times \mathcal{V}_2\,:\, V_1V_2=0\},$$
see \cite{FES}.

\smallskip

\begin{example} Let $\mathcal{V}_1$ be the set of circulant matrices and $\mathcal{V}_2$ the
set of diagonal matrices
in $\setC^{n \times n}$.
By Theorem \ref{sutu}, then $\mathcal{V}_1\mathcal{V}_2$ is closed.
It is readily seen that now the matrix product corresponds to 
the so-called Segre map,
a fundamental family of functions in 
algebraic geometry.
For the Segre maps, see, e.g., \cite[p.25]{HA}.
(In other words, 
the Kronecker product of two vectors
is really just an instance of the standard matrix
product restricted to prescribed matrix subspaces.)
The above theorem yields
a natural generalization of such  maps.
\end{example}

\smallskip

\begin{example}\label{ranki}
Certainly, the conditions of Theorem \ref{sutu} are not necessary. 
For instance, the
subset of $\setC^{n \times n}$ consisting of matrices of rank $k$ at most is 
closed, a property of tremendous importance in approximating
with the singular value decomposition. 
It equals the set of products $\mathcal{V}_1 \mathcal{V}_2$, where
$\mathcal{V}_1$ (resp. $\mathcal{V}_2$)
is the matrix subspace of $\setC^{n \times n}$ 
having the last $n-k$ columns (resp. rows) zeros.
\end{example}

\smallskip

The following example supports the viewpoint that the set of products
yields an equally natural ``discretization'' of Toeplitz
operators as Toeplitz  matrices do.  

\smallskip

\begin{example}\label{tope}
For an infinite dimensional noncommutative example,
for invertible Toeplitz operators there is an upper-lower triangular
factored Toeplitz structure; see, e.g., \cite{LAX}. 
For Toeplitz matrices this product does not preserve Toeplitzness.
The structure remains closed, though.
\end{example}

\smallskip

\begin{corollary} Denote by $\mathcal{V}_1$ and 
$\mathcal{V}_2$ the subspaces of upper and lower triangular 
Toeplitz matrices of $\setC^{n \times n}$ over $\setC$. Then
$\mathcal{V}_1\mathcal{V}_2$ is closed. 
\end{corollary}

\begin{proof} 
Consider the equation $V_1V_2=0$. In the product, compute the last row first to 
have either the diagonal of  $V_1$ zero or $V_2=0$. Then proceed analogously upwards
by computing next to the last row, to have the claim by using Theorem
\ref{sutu}. 
\end{proof}

The set of products in this corollary is intriguing since 
both subspaces are invertible.\footnote{Let $\mathcal{V}$ and $\mathcal{W}$ be two nonsingular matrix subspaces
of $\setC^{n \times n}$ over $\setC$ (or $\setR$). 
If $$
\{ V^{-1}\,:\,V\in\mathcal{V}\cap {\rm GL}(n,\setC)\}
=\mathcal{W}\cap {\rm GL}(n,\setC),$$
then  $\mathcal{V}$ is said to be invertible \cite{HR}.}
 Thereby it is a straightforward computation
to recover if a given matrix $A\in \setC^{n \times n}$ belongs
to $\mathcal{V}_1\mathcal{V}_2$ \cite{HR}. 
The closure of the set of inverses of invertible elements is the set of products
$\mathcal{V}_2\mathcal{V}_1$, which is closed  by analogous arguments.
Hence we have an 
elegant symmetry with respect to the inversion.\footnote{The problem
of characterizing the inverses of Toeplitz matrices has
been studied a lot; see, e.g., \cite{GO}.}

Although not readily determined,  the following
quantity appears to be of central relevance. To the best of our knowledge, it was initially introduced in \cite{BUSS}.
See also \cite{FAU} for related computations.

\begin{definition}
Let $\mathcal{V}$ be a matrix subspace of 
$\setC^{n \times n}$ over $\setC$ (or $\setR$).
Set $${\rm minrank}(\mathcal{V})=
\min_{V\in \mathcal{V},\,V\not=0}\,{\rm rank}(V).$$
\end{definition}

This quantity, let us call it the minrank of $\mathcal{V}$, 
is preserved under 
equivalence. In particular, in the nonsingular case of $\dim \mathcal{V}=2$,  
the geometric multiplicities of the eigenvalues of a matrix 
determine ${\rm minrank}(\mathcal{V})$. Namely, let
$\mathcal{V}={\rm span}\{V_1,V_2\}$ with $V_1,V_2\in \setC^{n \times n}$.
Suppose $X=I$ 
and $Y$ is an invertible element
of $\mathcal{V}$. Then 
\begin{equation}\label{ekvi}
X \mathcal{V}Y^{-1}= 
{\rm span}\{I,W_1\},
\end{equation} 
so that it suffices to inspect the eigenvalues of $W_1 \in \setC^{n \times n}$
to determine the minrank of $\mathcal{V}$.

\smallskip

\begin{example}\label{radonhu} The so-called Hurwitz-Radon 
matrix subspace $\mathcal{V}$ of 
$\setC^{n \times n}$ over $\setR$ has the property that
any nonzero element is a scalar multiple of a unitary matrix.
Hence, then we have ${\rm minrank}(\mathcal{V})=n$.
\end{example}

\smallskip

Using Proposition \ref{sutu} we
can conclude that, whenever 
$${\rm minrank}(\mathcal{V}_1)+{\rm minrank}(\mathcal{V}_2)>n,$$
then $\mathcal{V}_1\mathcal{V}_2$ is closed. 

The set of products 
need not be somehow curved
by the fact that its closure can retain the structure 
of matrix subspace.
Of course, we always have a subspace 
when either $\mathcal{V}_1$ or $\mathcal{V}_2$
is one dimensional.

\smallskip

\begin{example}\label{band}
 For a  matrix subspace admitting
a ``factorization'' as the product of two matrix subspaces,
consider the set of $(p,q)$-band 
matrices.\footnote{A $(p,q)$-band matrix is a square matrix with lower 
bandwidth $p$ and upper bandwidth $q$.}
Denote by $\mathcal{V}_1$ and $\mathcal{V}_2$ the set of $(p,0)$-band and 
$(0,q)$-band matrices.
Then there exists a factorization of a nonsingular $A\in \setC^{n \times n}$  as
$A=V_1V_2$
if and only if 
$A$ 
is a strongly nonsingular $(p,q)$-band matrix.
As strong singularity 
is a generic property, we can conclude that the closure of 
$\mathcal{V}_1\mathcal{V}_2$ 
is the set of $(p,q)$-band matrices, i.e., a matrix subspace.
\end{example}

\smallskip
 
Observe that between Examples \ref{ranki} and \ref{band} there is strong resemblance 
with regard to how the sparsity structures of the matrix subspaces are defined.

In statistics,  the so-called Craig-Sakamoto theorem\footnote{The Craig-Sakamoto theorem: 
Two real symmetric matrices  $X_1$ and $X_2$ satisfy 
$\det(I-tX_1-sX_2)= \det(I-tX_1) \det(I-sX_2)$ for all $t,s\in \setR$ if and only if
$X_1X_2=0$.}
 is similarly 
related with a factorization of a matrix subspace;
for details on the Craig-Sakamoto theorem, see \cite{DS}.
Besides statistics, this structure is of interest in  
studying 
the spectrum in the case of three dimensional 
matrix subspaces \cite{TA}.
It is, in essence, concerned with
the claim of the following proposition in the case 
of $k=2$ and $\dim \mathcal{V}_1=
\dim \mathcal{V}_2=1$.

\begin{proposition} Assume a matrix subspace $\mathcal{V}$ 
of $\setC^{n \times n}$ over $\setC$ can be
decomposed as 
$\mathcal{V}=\setC I +\sum_{j=1}^k\mathcal{X}_j$
with matrix subspaces $\mathcal{X}_j$ satisfying
$\mathcal{X}_j\mathcal{X}_{l}=\{0\}$ for $j<l$. Then
$\mathcal{V}$ equals the closure of $\prod_{j=1}^k(\setC I +\mathcal{X}_j).$
\end{proposition}

\begin{proof} 
Assume
$V=tI+X_1+\cdots +X_k\in \mathcal{V}$
 with $X_j \in \mathcal{X}_j$ and a scalar $t\not=0$.
(If $t=0$, then $V$ can be approximated with such an element.)
By the fact that $X_jX_l=0$ for $j<l$,  we can write
$V=(tI+X_1)(I+X_2/t)\cdots (I+X_k/t)$.
\end{proof}

The appearing condition on the matrix subspaces $\mathcal{X}_j$ 
forces them to be singular. 
It is noteworthy that this structure 
is also utilized
in computing (inverting) the 
$L$ factor in the LU decomposition of a matrix.

Let $\mathcal{V}_1$ and $\mathcal{V}_2$ be 
nonsingular matrix subspaces
of $\setC^{n \times n}$ over $\setC$ of dimension two.
(Due to the generalized eigenvalue problem, the two dimensional
case is possibly the most frequently encountered.) 
In view of the Craig-Sakamoto theorem, it is natural to ask when
the closure of  $\mathcal{V}_1\mathcal{V}_2$
is a matrix subspace. 
Here we may equally well consider equivalent matrix subspaces of the form 
\begin{equation}\label{tyoe}
X\mathcal{V}_1={\rm span}\{I,X_1\}\,
\mbox{ and }\,\mathcal{V}_2Y^{-1}={\rm span}\{I,X_2\}
\end{equation}
with invertible matrices
$X,Y \in \setC^{n \times n}$ chosen appropriately.
Then, instead of $\mathcal{V}_1\mathcal{V}_2$, it suffices to inspect
the set of products $X\mathcal{V}_1\mathcal{V}_2Y^{-1}$.

\begin{theorem}\label{jse} 
Let $\mathcal{V}_1$ and $\mathcal{V}_2$
be nonsingular matrix subspaces
of $\setC^{n \times n}$ over $\setC$ of dimension $2$. 
If the closure of   $\mathcal{V}_1\mathcal{V}_2$ is
a matrix subspace, then it is of dimension $3$ at most. 
This holds if and only if
the matrices in \eqref{tyoe} satisfy
\begin{equation}\label{frat}
X_1(cX_2-dI)=(aX_2-bI)
\end{equation}
for some constants $a,b,c,d\in \setC$  not all zero.
\end{theorem}

\begin{proof} Suppose 
${\rm span}\{V_1,V_2\}=\mathcal{V}_1$ and 
${\rm span}\{V_3,V_4\}=\mathcal{V}_2$.
For the claim we may consider the map 
\begin{equation}\label{just}
(z_1,z_2,z_3,z_4)\mapsto 
z_1z_3V_1V_3+z_1z_4V_1V_4+z_2z_3V_2V_3+z_2z_4V_2V_4
\end{equation}
from $\setC^4$ to $\setC^{n \times n}$. If the closure of
the image were four dimensional, then 
the set of equations \begin{equation}\label{jep} 
z_1z_3=b_1,\, z_1z_4=b_2,\,
z_2z_3=b_3\, \mbox{ and }\, z_2z_4=b_4
\end{equation} 
should have a solution for   
a dense subset of vectors $(b_1,b_2,b_3,b_4)$ of $\setC^4$.
Choose $b_1=t+\epsilon_1$, $b_2=t+\epsilon_2$,
$b_3=-t+\epsilon_3$ and $b_4=t+\epsilon_4$ with $t\in \setC\backslash \{0\}$
and small $|\epsilon_j| \ll |t|$, for $j=1,\ldots, 4$.
(For $\epsilon_1=\epsilon_2=\epsilon_3=\epsilon_4=0$
we clearly have no solutions.) We obtain
$$z_3=(1+\hat{\epsilon})z_4\, \mbox{ and }\, z_3=(-1+\tilde{\epsilon})z_4,
$$
where $\hat{\epsilon}=\frac{\epsilon_1-\epsilon_2}{t+\epsilon_2}$  
and $\tilde{\epsilon}=\frac{\epsilon_3-\epsilon_4}{t+\epsilon_4}$.
Subtracting leads to $0=(2+\tilde{\epsilon}- \hat{\epsilon} )z_4$
which is a contradiction for small small $\epsilon_j$. Hence there are
no solutions  near the point $(1,1,-1,1)$.


Assume the closure of $\mathcal{V}_1 \mathcal{V}_2$ is a 
matrix subspace.
Then the matrices 
$V_1V_3$, $V_1V_4$, $V_2V_3$ and $V_2V_4$ are contained in this subspace. 
(For example, the choices
$z_1=z_3=1$ and   
$z_2=z_4=0$ yield $V_1V_2$, etc.)
If the dimension is 3, it follows that 
$V_1V_3$, $V_1V_4$, $V_2V_3$ and $V_2V_4$ are linearly dependent,
i.e., $X_1(X_2-dI)=aX_2-bI$.
If the dimension is 2, then $X_1=aX_2-bI$.

For the converse, assume 
$\dim \{V_1V_3,V_1V_4,V_2V_3,V_2V_4\}=3.$ (The case of dimension
2 is trivial.) We may assume, let us say, 
$V_1V_3=\alpha_1V_1V_4+\alpha_2V_2V_3+\alpha_3 V_2V_4$
with fixed $\alpha_j \in \setC$, for $j=1,2,3$. Inserting this into \eqref{just} gives
$$(\alpha_1 z_1z_3+z_1z_4)V_1V_4+
(\alpha_2z_1z_3+z_2z_3)V_2V_3
+
(\alpha_3z_1z_3+z_2z_4)V_2V_4.$$
Let a vector $b=(b_1,b_2,b_3)$ be given. 
After possibly an arbitrarily small perturbation, we may assume
$b_2\not=0$. Then necessarily 
 $z_3\not=0$, so that with $w=\frac{z_4}{z_3}$ the problem of 
attaining these coefficients is equivalent to finding a solution to 
the equations
\begin{equation}\label{jupu}
z_1(\alpha_1+w)=c_1,\; \alpha_2z_1+z_2=c_2,\;
\alpha_3z_1+wz_2=c_3
\end{equation}
with $c_2\not=0$.
The last two equations, after eliminating $z_2$, can be combined 
into a single one 
$z_1(\alpha_3-\alpha_2w)+c_2w=c_3$. This combined with the
first equation in \eqref{jupu} results 
in $$c_2w^2+(c_2\alpha_1-c_1\alpha_2-c_3)w+c_2\alpha_3-c_3\alpha_1=0$$
which should have a nonzero solution satisfying $w\not=-\alpha_1$.
This can be achieved, after possibly an arbitrarily small perturbations
of the coefficients $c_1$, $c_2$ and $c_3$.
 \end{proof}

%

In particular, the Craig-Sakamoto theorem corresponds to the special case in which
the constants satisfy $a=b=d=0$ and $c=1$.

In the identity \eqref{frat} we are dealing with what we call 
a generalized linear fractional transformation. 
Namely, if $d/c$ is not an eigenvalue of $X_2$, then 
$X_1$ is a classical linear fractional transformation of $X_2$.
Combined with the equivalence transformations \eqref{tyoe}, this yields 
a way to construct two dimensional matrix subspaces such that the closure
of the set of products is a matrix subspace.

\subsection{The set of products and bihomogeneous 
polynomial maps of bidegree $(1,1)$} There
is a family of polynomial functions, 
studied especially  in algebraic geometry, 
which is intimately connected with
the product of two matrix subspaces. 
To describe this connection, 
for the set of products 
let us introduce its linearization defined as follows.

\begin{definition} Assume
$\mathcal{V}_1$ and $\mathcal{V}_2$ are matrix subspaces
of $\setC^{n \times n}$ over $\setC$ (or $\setR$).
The matrix subspace  of the smallest possible
dimension including
$\mathcal{V}_1\mathcal{V}_2$ is
said to be the linearization of  
$\mathcal{V}_1\mathcal{V}_2$.
\end{definition}

Assume $V^1,\ldots,V^{j}$ and $V^{j+1},\ldots,V^k$ 
are
bases of  $\mathcal{V}_1$ and $\mathcal{V}_2$, respectively.
If $W^1,\ldots,W^l$ is a basis of the linearization $\mathcal{W}$, 
then the  inclusion relation of the linearization 
implies that 
$$V^sV^t=\sum_{r=1}^l m^{st}_rW^r$$ for  some constants $m^{st}_r\in \setC.$
Thereby the product of arbitrary elements $V_1 \in\mathcal{V}_1$ and
$V_2 \in\mathcal{V}_2$ can be written as  
\begin{equation}\label{yhteys}
V_1V_2=\sum_{s=1}^jz_sV^s\sum_{t=j+1}^k w_tV^t
=\sum_{r=1}^lz^TM_rw W^r
\end{equation}
with $z=(z_1,\ldots,z_j)\in \setC^j$, $w=(w_{j+1},\ldots,w_k)\in \setC^{k-j}$ and  
matrices 
$M_r=\{m_r^{st}\}\in \setC^{j\times (k-j)}.$
In terms of this expansion, the problem converts into inspecting 
the bihomogeneous polynomial map
\begin{equation}\label{bili}
(z,w)\longmapsto M(z)w=
\smat{z^TM_1\\ \vdots\\ z^TM_l}w
\end{equation}
 of bidegree $(1,1)$
from $\setC^j\times \setC^{k-j}$ to $\setC^l$. 
(For computational algebraic geometric
aspects of such functions, see \cite{FES}.) 
Clearly, $M:\setC^j\rightarrow \setC^{l \times (k-j)}$
is linear with the $p$th column equaling $z^TM_p$ at 
$z\in \setC^j$. 

\smallskip 

\begin{example}
In the proof of Theorem \ref{jse},  when the closure of 
$\mathcal{V}_1\mathcal{V}_2$ is the three dimensional
matrix subspace $\mathcal{W}$,
we have $j=2$, $k=4$ and $l=3$.
Then  $W=V_1V_2 $, $W_2=V_2V_3$ and $W=V_2V_4$
with $$M_1=\smat{\alpha_1&1 \\0&0},\,
M_2=\smat{\alpha_2&0 \\1&0}\,\mbox{ and }\,M_3=\smat{\alpha_3&0 \\0&1}.$$
\end{example} 

\smallskip

The following theorem is of importance.

\begin{theorem}\label{thanks} 
Let $\mathcal{V}_1$, $\mathcal{V}_2$ and $\mathcal{W}$ be matrix subspaces of  
$\setC^{n \times n}$ over $\setC$.
If the closure of 
$\mathcal{V}_1\mathcal{V}_2$ is $\mathcal{W}$,
then $\mathcal{V}_1\mathcal{V}_2$ contains an open dense subset of $\mathcal{W}$. 
\end{theorem} 

\begin{proof}  Using arguments from algebraic geometry, 
this follows directly from a theorem of Chevalley, i.e., 
from the fact that the map \eqref{bili} is regular
and that its image is therefore constructible. 
(For more details, see, e.g., \cite[Theorem 10.2]{Milne}
or \cite[p. 94]{HART}.)
If the closure of the image (in the standard topology of $\setC^l$)
is $\setC^l$, then $\setC^l$ is also the Zariski closure
of the image.
Thus, it contains a Zariski open  set of $\setC^l$. Such a set is open and dense in $\setC^l$. 
\end{proof}


In principle, with \eqref{bili}  the problem of recovering
 whether a given matrix $A\in \mathcal{W}$ is factorizable 
can be approached with the Hilbert Nullstellensatz. As opposed to
considering \eqref{taas}, this problem belongs to the
realm of commutative algebra. For the existence of 
a solution, by invoking the effective Nullstellensatz 
we have 
the linear algebra problem of solving a linear system.
It is, however, only formally so
because of the exponential growth of the size of the linear systems;
see Appendix. 

The converse problem is actually of interest. Namely, consider solving a bilinear system 
\begin{equation}\label{biliyh}
M(z)w=b
\end{equation}
with a given $b\in \setC^l$.  If an expansion
\eqref{yhteys} can be established, for some matrix subspaces $\mathcal{V}_1$, 
$\mathcal{V}_2$ and $\mathcal{W}$, with either  $\mathcal{V}_1$ or
$\mathcal{V}_2$ invertible
then solving \eqref{biliyh} is straightforward
with the algorithms proposed in \cite{HR}.
Consequently, the problem of finding such matrix subspaces for a 
given bilinear system
\eqref{biliyh} seems to be of central relevance.

\section{Riemannian geometry of the product of matrix subspaces}\label{GPMS}
Next the geometry of the set of products is inspected. The flat case is 
introduced first. Then the general case is studied in terms of smooth maps.

\subsection{Factorizable matrix subspaces}
The structure appearing in connection with the LU decomposition and the Craig-Sakamoto theorem  
is really just an instance of the following general notion of noncommutative
factoring.

\begin{definition}\label{lao} 
A matrix subspace $\mathcal{W}$ of 
$\setC^{n \times n}$ over $\setC$ (or $\setR$)
is said to be factorizable  
if 
\begin{equation}\label{ehtotu}
\mathcal{W}=\overline{\mathcal{V}_1\mathcal{V}_2}
\end{equation}
for matrix subspaces $\mathcal{V}_1$ and $\mathcal{V}_2$
of $\setC^{n \times n}$ over $\setC$ (or $\setR$)
satisfying the conditions
\begin{equation}\label{kondi} 
1<\min \{\dim \mathcal{V}_1,\dim \mathcal{V}_2\} \; \mbox{ and }\, 
\max \{\dim \mathcal{V}_1,\dim \mathcal{V}_2\}<\dim \mathcal{W}.
\end{equation}   
\end{definition}

A matrix subspace which is not factorizable is said to be irreducible.
As Example \ref{band} illustrates, 
taking the closure may be necessary. Observe though that, 
for matrix subspaces over $\setC$, the set of products 
$\mathcal{V}_1\mathcal{V}_2$ is topologically and measure 
theoretically large in $\mathcal{W}$
by Theorem \ref{thanks}.

The structure is preserved under equivalence, i.e., $\mathcal{W}$
is factorizable if and only if $X\mathcal{W}Y^{-1}$
is factorizable for any invertible $X,Y\in \setC^{n \times n}$.

A central problem,
typically tough for one reason or another, 
is to recover whether a given matrix subspace is
factorizable. 
Needless to say, if a matrix subspace $\mathcal{W}$ can be factored,
the interest turns completely 
on the factors $\mathcal{V}_1$ and $\mathcal{V}_2$.\footnote{This
is manifested by the LU factorization. After computing an LU
factorization, the original matrix is typically thrown away and only
the factors are saved in the storage.}

\smallskip

\begin{example} This is Example \ref{tope} continued.
In finite dimensions,
an analogous problem corresponds to asking if the set of Toeplitz
matrices is a factorizable matrix subspace. 
(It is likely of use to recall that Toeplitz matrices are equivalent to 
Hankel matrices.)
\end{example}

\smallskip

It is noteworthy that when the set of products is closed,
the singular elements of $\mathcal{W}$
are exactly determined by the singular elements of 
the factors $\mathcal{V}_1$ and $\mathcal{V}_2.$ 
(Otherwise only an inclusion can be guaranteed.) Hence, then the study of
the spectrum of $\mathcal{W}$, i.e., its singular elements, reduces to the study of the spectra of $\mathcal{V}_1$
and $\mathcal{V}_1$.
This is precisely the case in the
Craig-Sakamoto theorem.

\smallskip

\begin{example}\label{symme}
Interpreting the spectrum algebraic geometrically, in the case of $\mathcal{W}=\mathcal{V}_1\mathcal{V}_2$ 
the determinantal variety of $\mathcal{W}$ is determined
by the determinantal varieties of $\mathcal{V}_1$ and $\mathcal{V}_2$.
This ``factorization'' is 
intriguing since the 
determinantal variety of $\mathcal{W}$ may well be irreducible.
For example, we have $\setC^{n \times n}=\mathcal{V}_1\mathcal{V}_2$,
where $\mathcal{V}_1=\mathcal{V}_2$ is the set of symmetric matrices.
(This is a classical result; see, e.g., \cite{PHA} and \cite{LA}
 and references therein.)
\end{example}

\smallskip

In general, \eqref{ehtotu} is a noncommutative 
notion with respect to the factors.
(Curiously, when  $\dim \mathcal{V}_1=\dim \mathcal{V}_2
=1$, the commutative case 
corresponds to the $\omega$-commutativity 
of matrices \cite{HMS}.)
Although more stringent, in group theory 
there is the so-called internal Zappa-Sz\'ep product 
of a group having some 
aspects in common with Definition \ref{lao}.\footnote{Let $G$ be a 
group with the identity element $e$, and let $H$ and $K$ 
be subgroups of $G$. If
$G = HK$ and $H\cap K = \{e\}$, then $G$ is the internal Zappa-Sz\'ep product
of $H$ and $K$.}
For the the internal Zappa-Sz\'ep product, see the paper \cite{BRIN}, 
which is partially expository, 
and references therein.

Only the case where the conditions \eqref{kondi}
are satisfied  
is of interest since
otherwise we are dealing with the equivalence of matrix subspaces. 
Consequently,  a two dimensional
matrix subspace   is never factorizable. 
Then, through the equivalence \eqref{ekvi}, the nonsingular case is 
completely understood
in terms of canonical forms for matrices. Two dimensional
matrix subspaces are hence fundamentally different and correspond,
in essence, to classical matrix analysis. 
Observe that the first nontrivial case, i.e.,
the three dimensional case over $\setC$ 
can be regarded as understood by Theorem \ref{jse}.

In Example \ref{ranki} there is the growth of $k$ and
in Example \ref{band} the growth of the bandwidth,
i.e., we have natural hierarchies of
sets of products. The maximum values of the indices
correspond to the set of products $\setC^{n \times n}$.
Analogously, suppose \eqref{ehtotu} holds.
Take two sequences of nested subspaces
$$\mathcal{V}_1^1\subset\mathcal{V}_1^2\subset
\cdots\subset \mathcal{V}_1^{\dim \mathcal{V}_1}= \mathcal{V}_1\, 
\mbox{ and }\, 
\mathcal{V}_2^1\subset\mathcal{V}_2^2\subset
\cdots \subset\mathcal{V}_2^{\dim \mathcal{V}_2}= \mathcal{V}_2.$$ 
Then $\{\mathcal{V}_1^j\mathcal{V}_2^k\}_{j,k}$ yields a natural hierarchy of
sets of products such that for the maximum values of the indices
the closure of the set of products is $\mathcal{W}.$


After these general remarks, 
let us give further illustrations on how Definition \ref{ehtotu}
actually encompasses numerous
classical instances.

\smallskip

\begin{example}\label{tarkea} 
Although we are concerned with finite dimensional
(noncommutative matrix) subspaces, 
infinite dimensional (commutative) problems are 
certainly of equal interest.
For the classical problem of integer factoring, let $k\in \setZ$ be given.
Associate with $k$ the set $k \setZ$, regarded as subspace of $\setZ$ 
over $\setZ$. 
Then
the question of whether $k$ is
a irreducible is equivalent to asking whether $k\setZ$ is
a factorizable subspace.
\end{example}

\smallskip

%
%
%

For another commutative case, 
consider the ring of complex polynomials. To this 
corresponds a very  classical notion.
Namely, denote by $\mathcal{V}_k$ the subspace consisting of
complex polynomials of degree $k$ at most.
(Of course, factoring depends
heavily on the choice of field.)  
Then the fundamental
theorem of algebra can be stated as a factorization result 
$$\mathcal{V}_k=\mathcal{V}_1\mathcal{V}_{k-1}$$
for $\mathcal{V}_k$. 
(This can also be formulated as an operator factorization result
in infinite dimensions
for lower triangular Toeplitz operators with finite
bandwidth.)
 Of course, this is a tremendously powerful fact
with practical implications, i.e., one would certainly 
like to have any given polynomial factored.

\smallskip

\begin{example}\label{tarkeahko} Polynomial factoring relates to matrix subspaces
as follows.
Suppose $A\in \setC^{n \times n}$ and consider
$$\mathcal{V}=\mathcal{K}(A;I)={\rm span}_{\setC}\,\{I,A,A^2,\ldots\}.$$
Then, by the fundamental
theorem of algebra, $\mathcal{V}$
is a factorizable matrix subspace whenever ${\deg}(A)>2$.
\end{example}


\subsection{Geometry of the product of matrix subspaces}
In general, the set of products of two matrix subspaces  $\mathcal{V}_1$
and $\mathcal{V}_2$ of 
$\setC^{n \times n}$ over $\setC$ (or $\setR$)
cannot be expected to be flat, i.e., to yield a factorizable matrix subspace.
The study of its geometry 
involves  linear  structures, though.
To this end, associate 
with the set of products
the bilinear map
\begin{equation}\label{tiis}
\Psi(V_1,V_2)=V_1V_2
\end{equation}
from the direct sum $\mathcal{V}_1\times \mathcal{V}_2$ 
to $\setC^{n\times n}$. Certainly, $\Psi$ is smooth as it 
 can be treated as a bihomogeneous polynomial map 
once a basis
of the vector space
$\mathcal{V}_1\times \mathcal{V}_2$ has been fixed
appropriately; see \eqref{bili}.

\smallskip

\begin{example} Although not done in this paper,  
let us emphasize that it is also natural to study
$\Psi$ on subspaces of $\mathcal{V}_1\times \mathcal{V}_2$.
For example, take $\mathcal{V}_1=\mathcal{V}_2^T$  to be the
set of lower triangular matrices in $\setC^{n \times n}$. Then, 
related with the LU factoring of symmetric matrices,
consider $\Psi$ on the
subspace $\{(V_1,V_2)\in \mathcal{V}_1\times \mathcal{V}_2
\,:\,V_1=V_2^T\}$. 
\end{example}

\smallskip


If the linearization of $\mathcal{V}_1\mathcal{V}_2$ is not the whole 
$\setC^{n \times n}$, then the set of products 
is said to be degenerate.
In view of this, for matrix subspaces 
of $\setC^{n \times n}$ over $\setC$, a generalization of Picard's theorem
\cite{GR} can be used to 
bound the dimension of the linearization.
For this, assume $A\in \setC^{n\times n}$ and look at the map $A-\Psi$ defined as 
\begin{equation}\label{taas}
(V_1,V_2)\longmapsto 
A-V_1V_2
\end{equation}  
If $A\not\in \mathcal{V}_1\mathcal{V}_2$, then
this does not have zeros and hence we may
 apply the following theorem with 
$m=\dim \mathcal{V}_1+
\dim \mathcal{V}_2$ and $l=n^2-1$.

\begin{theorem} \cite{GR} Let $f:\setC^m\rightarrow \setP_l(\setC)$ be
 a holomorphic map that omits $l+k$ hyperplanes\footnote{A hyperplane 
in homogeneous coordinates $z_0,\ldots, z_l$ is
$\{(z_0,\ldots,z_l)\,:\, \sum_{j=0}^la_jz_j=0\}$
for fixed $a_j\in \setC$.} 
in general position, $k \geq 1$.
Then the image of $f$ is contained in a projective linear subspace
of dimension $\leq [l/k]$, where the brackets mean greatest integer.
\end{theorem}

To inspect the structure of the set of products
more locally,
on any matrix subspace $\mathcal{V}$ of $\setC^{n \times n}$
over $\setC$
the standard inner product
\begin{equation}\label{inne}
(V_1,V_2)={\rm tr}(V_2^*V_1)\, \mbox{ with }\, V_1,V_2 \in \mathcal{V}
\end{equation}
is used. (Take the real part for matrix subspaces over $\setR$.)
The respective norm $||\cdot||_F$ is  the Frobenius norm.
To proceed Riemannian geometrically, take 
a smooth curve with the Taylor expansion 
$$c(t)=(V_1,V_2)+t(W_1,W_2)+
t^2(U_1,U_2)+\cdots$$
in $\mathcal{V}_1\times \mathcal{V}_2$ passing through 
a point $(V_1,V_2)$.
Then $\Psi$ maps this curve to $\setC^{n \times n}$ as
\begin{equation}\label{kurvi}
V_1V_2+t(V_1W_2+W_1V_2)+t^2(V_1U_2+W_1W_2+U_1V_2)+\cdots.
\end{equation}
Considering its linearization, the linear terms
span the matrix subspace 
\begin{equation}\label{tangent} 
V_1\mathcal{V}_2+\mathcal{V}_1V_2.
\end{equation}
In particular, the dimension of this subspace yields 
the rank of $\Psi$ at $(V_1,V_2)$. 
%

%
%
%

It is clear that 
the maximum rank yields a lower bound on the dimension
of the linearization of $\mathcal{V}_1\mathcal{V}_2$. 
(Proof: \eqref{tangent} is a subset of the linearization
of $\mathcal{V}_1\mathcal{V}_2$.)

\begin{proposition}\label{apupa}
Let $\mathcal{V}_1$ and $\mathcal{V}_2$ be two subspaces of 
$\setC^{n \times n}$ over $\setC$ (or $\setR$).
Then $\overline{\mathcal{V}_1\mathcal{V}_2}$
equals its linearization if and only if $\overline{\mathcal{V}_1\mathcal{V}_2}$
contains the matrix subspace \eqref{tangent} for any $(V_1,V_2)$.
\end{proposition}

\begin{proof} Suppose $\overline{\mathcal{V}_1\mathcal{V}_2}$
is a subspace, i.e., equals its linearization. 
Since $\Psi :\mathcal{V}_1 \times \mathcal{V}_2 \rightarrow
\overline{\mathcal{V}_1\mathcal{V}_2}$,
it follows that $\overline{\mathcal{V}_1\mathcal{V}_2}$
contains the matrix subspaces \eqref{tangent}.

Suppose $\overline{\mathcal{V}_1\mathcal{V}_2}$
contains  the matrix subspaces \eqref{tangent}. 
Then $V_1Y+XV_2\in \overline{\mathcal{V}_1\mathcal{V}_2}$
for any $V_1,X\in \mathcal{V}_1$ and $Y,V_1\in \mathcal{V}_2$,
i.e., it contains the sums. The multiples are contained by the
homogeneity. Thereby $\overline{\mathcal{V}_1\mathcal{V}_2}$  is a subspace.
\end{proof}

The following facts are well known; see, e.g., \cite[Chapter 2]{CON}.

\begin{proposition}
The rank of $\Psi$ 
generically attains the maximum.
\end{proposition}

\begin{proof}
Consider the derivative of $\Psi$ regarded as  
 a function of $m=\dim \mathcal{V}_1+
\dim \mathcal{V}_2$ variables after fixing 
bases of $\mathcal{V}_1$ and $\mathcal{V}_2$.
Recall that  a
 matrix has rank $r$ if and only if
 there exists at least one non-zero $r$-by-$r$
minor. Computing the determinant of the respective minor 
of the derivative gives
a nonzero polynomial in $m$  variables. 
The points where the maximum of the
rank is not attained belong to its zero set.
\end{proof}

The following also justifies calling \eqref{tangent}
locally the tangent  space of
$\mathcal{V}_1\mathcal{V}_2$ at $\Psi(V_1,V_2)=V_1V_2$.

\begin{proposition}
Let the
rank of $\Psi$ attain the maximum at $(V_1,V_2)$.
Then the image of a neighbourhood of $(V_1,V_2)$ under  $\Psi$ 
is a smooth submanifold of $\setC^{n \times n}$ 
of the dimension equaling the rank.
\end{proposition}

\begin{proof} Let $m=\dim \mathcal{V}_1+\dim \mathcal{V}_2$
and denote by $k$ rank of $\Psi$ at $(V_1,V_2)$.
Because of the constant rank theorem,
we have the normal form 
$$\psi \circ \Psi \circ \phi^{-1} (x_1,\ldots, x_{m})
\mapsto (x_1,\ldots,x_k,0,0,\ldots,0)$$
for appropriate charts $\phi$ and $\psi$.
Denote the components of $\psi$ by $\psi_j$
and consider the map $(\psi_{k+1},\ldots, \psi_{n^2})$
from $\setC^{n \times n}$ to $\setC^{n^2-k}$. 
For this map, look at the inverse
image of $(0,0,\ldots,0)$ to have the claim (a consequence of the 
constant rank theorem).
\end{proof}

Assume the rank of $\Psi$ attains its maximum at $(V_1,V_2)$.
By the above proposition, locally the
image can be regarded  as a smooth submanifold 
of $\setC^{n \times n}$. 
With respect to the standard inner product on $\setC^{n \times n}$, denote by 
$\mathbf{P}_{V_1,V_2}$ the orthogonal
projector on $\setC^{n \times n}$ onto 
the tangent space $V_1\mathcal{V}_2+\mathcal{V}_1V_2$.
Then to Riemannian geometrically
assess how curved $\mathcal{V}_1\mathcal{V}_2$ is 
at $\Psi (V_1,V_2)$,
consider the map
\begin{equation}\label{mittari}
Q_{(V_1,V_2)}(W_1,W_2)=2(I-\mathbf{P}_{V_1,V_2})W_1W_2
\end{equation}
on $\mathcal{V}_1\times \mathcal{V}_2$.
Using this notation, its introduction can be argued as follows.

\begin{proposition} Let the rank of $\Psi$ attain its maximum 
at $(V_1,V_2)$. Then, 
for the image of a neighbourhood of $(V_1,V_2)$ under  $\Psi$,  
the extrinsic curvature of the geodesic 
passing through 
$\Psi(V_1,V_2)$ with the speed vector $V_1W_2+W_1V_2$ equals  
\begin{equation}\label{excurv}
\left|\left| Q_{(V_1,V_2)}(W_1,W_2) \right|\right|_F.
\end{equation}
\end{proposition}

\begin{proof}
An application of \eqref{mittari} always corresponds to
a curve \eqref{kurvi} having the acceleration 
orthogonal to  $V_1\mathcal{V}_2+\mathcal{V}_1V_2$  
which is achieved by choosing the coefficient $(U_1,U_2)$
such that $V_1U_2+W_1W_2+U_1V_2$ belongs to the orthogonal complement of 
the matrix subspace $V_1\mathcal{V}_2+\mathcal{V}_1V_2$.
Such a property is required from the geodesics
passing through $\Psi(V_1,V_2)$; see \cite[pp. 138--139]{lee}.
Hence, for a geodesics passing through 
$\Psi(V_1,V_2)$ with the speed vector $V_1W_2+W_1V_2$, the 
coefficient $(U_1,U_2)$ is determined by this condition.
\end{proof}

If the map \eqref{mittari} vanishes identically, then the set of products
$\mathcal{V}_1\mathcal{V}_2$ belongs to $V_1\mathcal{V}_2+\mathcal{V}_1V_2$.
Then also the linearization of $\mathcal{V}_1\mathcal{V}_2$
equals this matrix subspace.

Geodesics and measuring curvature are closely related. Namely, 
the second fundamental form ${\rm II}$ 
for the image of $\Psi$ in a neighbourhood of $(V_1,V_2)$
can  be found by setting 
$\frac{1}{2}(Q_{(V_1,V_2)}(W_1+\tilde{W}_1,W_2+\tilde{W}_2)-
Q_{(V_1,V_2)}(W_1,W_2)-Q_{(V_1,V_2)}(\tilde{W}_1,\tilde{W}_2))$ to be 
$${\rm II}(W_1,W_2,\tilde{W}_1,\tilde{W}_2)=
(I-\mathbf{P}_{V_1,V_2})(W_1\tilde{W}_2+\tilde{W}_1W_2).$$
For more details on the second fundamental form
and its geometric interpretation, see \cite[p. 138]{lee}.

As a first example, Theorem \ref{jse} corresponds to vanishing
curvature. For another, with nonvanishing curvature, 
 consider the subset of matrices related with approximating
with the singular value decomposition.

\smallskip

\begin{example} This is Example \ref{ranki} continued, i.e., 
consider $\mathcal{F}_k \subset \setC^{n \times n}$,  
the set of matrices of rank $k$ at most.
Take $(V_1,V_2)$ to be a generic point of
$\mathcal{V}_1\times \mathcal{V}_2$ and 
let $V_1=U_1\Sigma_1W_1^*$
and $V_2=U_2\Sigma_2W_2^*$ be the singular value decompositions
of $V_1$ and $V_2$.\footnote{For $j=1,2$, the matrices $U_j\in \setC^{n \times n}$ and 
$W_j\in \setC^{n \times n}$ are unitary.  In the matrices $\Sigma_j$ only the first $k$ diagonal entries 
can be nonzero. For  more details on  the singular value decomposition, see \cite{GV}.}
Then at $\Psi(V_1,V_2)$ we have 
$$V_1\mathcal{V}_2+\mathcal{V}_1V_2=U_1\mathcal{W}W_2^*,$$
where $\mathcal{W}$ is the matrix subspace of $\setC^{n \times n}$ consisting of matrices whose
first $k$ rows and columns can be freely chosen. Hence, the points 
$(W_1,W_2)\in \mathcal{V}_1\times \mathcal{V}_2$
satisfying $Q_{(V_1,V_2)}(W_1,W_2)=2W_1W_2$ are readily determined. Consequently, 
$\mathcal{V}_1\mathcal{V}_2=\mathcal{F}_k$
can be regarded as being maximally curved when measured in terms of \eqref{excurv}.
It is of interest to note that 
$\mathcal{V}_2\mathcal{V}_1$ is flat.
\end{example}

\smallskip

Observe that although $Q_{(V_1,V_2)}$ yields a local measure of curvature,
for the image of $\Psi$ around the point $\Psi(V_1,V_2)$, it  
possesses a global character by the
fact that the appearing orthogonal projector  operates on the full image of $\Psi$.
This provides an opportunity to use local information to draw conclusions
about global properties
of the set of product as follows.

\begin{theorem}\label{juku}
Let $\mathcal{V}_1$ and $\mathcal{V}_2$ be matrix subspaces of 
$\setC^{n \times n}$ over $\setC$. 
Then $\overline{\mathcal{V}_1\mathcal{V}_2}$ equals
its linearization if and only if \eqref{mittari} vanishes
for some $(V_1,V_2).$
\end{theorem}

\begin{proof} Assume $\overline{\mathcal{V}_1\mathcal{V}_2}$
equals its linearization $\mathcal{W}$. 
By Sard's theorem, the set of critical values of  $\Psi$ is 
of the first category with Lebesque measure zero; see, e.g., 
\cite[p. 260, Theorem 1]{POS}. 
However, by Theorem \ref{thanks}, the image contains an
open subset of $\mathcal{W}$. Its Lebesque measure is obviously positive. 
Thereby there must be points
in the image which are regular values. Consequently, 
\eqref{mittari} vanishes
for some $(V_1,V_2).$

For the converse, suppose \eqref{mittari} vanishes
for some $(V_1,V_2).$ 
Clearly, then the closure of $\mathcal{V}_1\mathcal{V}_2$ belongs to the respective matrix subspace
$\mathcal{W}=V_1\mathcal{V}_2+
\mathcal{V}_1V_2$, i.e., the linearization of  $\mathcal{V}_1\mathcal{V}_2$. 
We may regard $\Psi :\mathcal{V}_1\times \mathcal{V}_2 \rightarrow \mathcal{W}.$  
By the constant rank theorem, 
the image of a neighbourhood of $(V_1,V_2)$ under  $\Psi$ 
contains an open set (in the standard
topology) in $\mathcal{W}$. 
In the Zariski topology, the image of $\Psi$ is constructible and
contains an open subset of its closure. 
This follows from a theorem of Chevalley.
(For more details, see, e.g., \cite[Theorem 10.2]{Milne}
or \cite[p. 94]{HART}.)
Because the image of $\Psi$
contains an open subset in the standard topology, the 
open subset of the image of $\Psi$ in the Zariski topology
is dense in $\mathcal{W}$. 
\end{proof}

Obviously, under the assumptions of this theorem,   $\mathcal{W}=\overline{\mathcal{V}_1\mathcal{V}_2}$ is factorizable (assuming \eqref{kondi}
is satisfied).

An immediate example is given by the
LU factorization.

\smallskip

\begin{example}
Using Theorem \ref{juku}, it is straightforward to 
conclude 
that an LU factorization exists
for elements in a dense open subset of $\setC^{ n \times n}$. 
(Of course, this
is well known; a nonsingular matrix can be LU factored if and only if it
is strongly nonsingular \cite[p. 162]{HJ}.)
Namely, let $\mathcal{V}_1$ 
denote the set of lower triangular matrices 
and $\mathcal{V}_2$ 
upper triangular matrices with constant diagonal.
Then at $(I,I)\in \mathcal{V}_1\times\mathcal{V}_2$ 
the rank of $\Psi$ is readily seen to be $n^2.$
Consequently, \eqref{mittari} vanishes, so that
$\setC^{n \times n}=\overline{\mathcal{V}_1\mathcal{V}_2}$.
\end{example}

\smallskip

Of course, we may view this example just as a special case of the following
consequence of Theorem \ref{juku}. 

\begin{corollary}
Suppose $\mathcal{V}_1$ and $\mathcal{V}_2$ are matrix
subspaces of $\setC^{n \times n}$ over $\setC$ both containing the identity.
If $\setC^{n\times n}=\mathcal{V}_1+\mathcal{V}_2$, then 
$\setC^{n \times n}=\overline{\mathcal{V}_1\mathcal{V}_2}$.
\end{corollary}

Analogously can be formulated claims for lower (upper)
triangular matrices
in case of subspaces $\mathcal{V}_1$ and $\mathcal{V}_2$
of lower (upper) triangular matrices as follows.

\begin{corollary}
Suppose $\mathcal{W}$ is equivalent to 
a subalgebra 
of $\setC^{n \times n}$ over $\setC$ containing 
invertible elements.
Then $\mathcal{W}$ is factorizable.
\end{corollary} 

\begin{proof} We may assume $\mathcal{W}$ is
a subalgebra 
of $\setC^{n \times n}$ over $\setC$ containing 
invertible elements. 
Observe that $\mathcal{W}$ contains the identity.
To see this, assume $A\in \mathcal{W}$ is invertible. Then
take a polynomial satisfying $p(A)=A^{-1}$.
We have $I=Ap(A)\in\mathcal{W}$.

Denote by $d$ the dimension of $\mathcal{W}$.
Take a basis $W_1, \ldots, W_d$ of $\mathcal{W}$ and 
set $\mathcal{V}_1={\rm span}\{I,W_1,W_2,\ldots,W_{k}\}$ and 
$\mathcal{V}_2=
{\rm span}\{I,W_{k+1},W_{k+2},\ldots,W_{d}\}$.
Clearly, $\mathcal{V}_1\mathcal{V}_2\subset \mathcal{W}$.
At  $(I,I)$ the rank of $\Psi$ is $d.$
\end{proof}

For a classical factorization,
any matrix $A\in \setC^{n \times n}$
is the product of two symmetric matrices such that there
are at least $n$ degrees of freedom to construct a
factorization \cite{HR}. (See also Example \ref{symme}.)
This redundancy can be reduced as follows.

\begin{proposition} Let  $\mathcal{V}_1$ 
be the set of symmetric matrices and 
$\mathcal{V}_2$ the subset of symmetric matrices having constant 
antidiagonal. Then $\setC^{n \times n}=\overline{\mathcal{V}_1\mathcal{V}_2}$.
\end{proposition}

\begin{proof}  Suppose $n$ is odd and denote by $J$ the permutation matrix having
ones on its antidiagonal. 
Then, for a generic diagonal matrix $D={\rm diag}(d_1,d_2,\ldots,d_n)$,  
the rank of $\Psi$ 
at $(D,J)\in \mathcal{V}_1\times\mathcal{V}_2$ 
can be shown to equal $n^2.$ (For $n$ even, proceed similarly with
$(J,D)\in \mathcal{V}_1\times\mathcal{V}_2$.)
To see this, consider $DV_2+V_1J$. Clearly, the $(j,k)$ entry of
$V_1J$ is the $(j,n-k+1)$ entry of $V_1$. For $k\not=j$ 
and $k\not=n-j+1$
this means that
the entries in $DV_2+V_1J$ located symmetrically with respect to
the diagonal and antidiagonal are interdependent, i.e.,
the $(j,k)$, $(k,j)$, $(n-k+1,n-j+1)$ and 
$(n-j+1,n-k+1)$ entries.  To satisfy
$D\mathcal{V}_2+\mathcal{V}_1J=\setC^{n\times n}$
yields us the linear system
$$
\smat{
d_j& 0 & 1 & 0\\
d_k& 0& 0 & 1\\
0&d_{n-k+1}& 1 & 0\\
0&d_{n-j+1}&0 & 1}
\smat{
V_2(j,k)\\
V_2(n-k+1,n-j+1)\\
V_1(j,n-k+1)\\
V_1(k,n-j+1)}
=
\smat{
M(j,k)\\
M(k,j)\\
M(n-k+1,n-j+1)\\
M(n-j+1,n-k+1)}
$$
which should have a solution for any right-hand side.
This is possible if and only if 
$1-\frac{d_kd_{n-k+1}}{d_jd_{n-j+1}}\not=0$ for $j\not=k$
and $k\not=n-j+1$.
For $k=j$ and $k=n-j+1$
we obtain the condition $d_j\not=0$.
These are the genericity conditions the matrix $D$ needs to satisfy.
\end{proof}

Since $\mathcal{V}_1$ is an invertible matrix subspace,
the factorization of the proposition 
can be computed with the methods proposed in 
\cite{HR}.




\section{Conclusions}
Motivated by factorization problems, the set of products of two matrix
subspaces is studied. 
Differential geometric approach
applied to the respective smooth mapping yields a measure
of curvature for the set of products. Its vanishing
corresponds to the concept of factorizable matrix subspace. 
The notion of irreducible matrix subspace was introduced.
The LU factorization, the singular value decomposition and 
the Craig-Sakamoto theorem
served as illustrative examples how
seemingly different concepts (algorithmically at least)
can be put under the same caption.

\section*{Appendix: Hilbert Nullstellensatz}
Let $p_1,\ldots, p_k$ be complex polynomials in $n$ variables.
There is no common zero in $\setC^n$ if and only if
there are complex polynomials $q_j$ in $n$ variables satisfying
$$\sum_{j=1}^kq_jp_j=1.$$
This is the Hilbert Nullstellensatz. 
The effective Hilbert Nullstellensatz 
states that the
degrees of the $q_j$ may be assumed to satisfy
$${\rm deg}\; q_j\leq 
\left\{ \begin{array}{cc} 
D^n&  \mbox{ if }\; D\geq 3\\
2^{\min\{n,k\}}& \mbox{ if }\; D=2\\
\end{array} \right. ,$$
where $D=\max \deg p_j.$
See \cite{BR}, \cite{KO} and \cite{SO}.


\begin{thebibliography}{1}










\bibitem{BRIN} {\sc  M. Brin},
{\em On the Zappa-Sz\'ep product}, Comm. Algebra, 33 (2005), pp. 393--424. 

\bibitem{BR} {\sc W. Brownawell,} 
{\em Bounds for the degrees in the Nullstellensatz,}
Ann. Math., 126  (1987), pp. 577--591. 


\bibitem{BUSS}
{\sc J. Buss, G. Frandsen and J. Shallit,} {\em The computational complexity of some problems of linear algebra,} 
J. Comput. System Sci., 58 (1999), pp. 572--596.

\bibitem{BHU} {\sc M. Byckling and M. Huhtanen}, {\em
Approximate factoring of the inverse}, 
Numer. Math., 117 (2011), pp. 507--528.


\bibitem{CON} {\sc L. Conlon}, Differential Manifolds, a First Course,
{\em Birkh\"auser, Boston}, 1993.

\bibitem{CLO} {\sc D. Cox,  J. Little and  D. O'Shea,} 
Ideals, Varieties and Algorithms. An Introduction to Computational 
Algebraic Geometry and Commutative Algebra, 
{\em Springer-Verlag, New York,} 1997.

\bibitem{CR}
{\sc C.W. Curtis and I. Reiner},
Representation Theory of Finite Groups and Associative Algebras, 
{\em AMS Chelsea Publishing}, 1962.



\bibitem{DS}
{\sc M. Dumais ans G.P. Styan}, 
{\em A bibliography on the distribution of quadratic forms in normal
variables, 
with special emphasis on the Craig-Sakamoto theorem and on 
Cochran's theorem,}
In G. Styan, ed., 
 Three Bibliographies and a Guide,
Seventh International Workshop on Matrices and Statisics
Fort Lauderdale, 1998, pp. 1--9. 


\bibitem{EFF}
{\sc E.G. Effros, Z.-J. Ruan}, Operator Spaces,
{\em Oxford University Press, Oxford}, 2000.

\bibitem{FES}
{\sc J.-C. Faug\`ere, M. El Din and P.-J. Spaenlehauer,}
{\em Gr\"obner bases of bihomogeneous ideals generated by 
polynomials of bidegree (1,1): algorithms and complexity,} 
J. Sym. Comp., 46 (2011), pp. 406--437.

\bibitem{FAU}
{\sc J.-C. Faug\`ere, M. El Din and P.-J. Spaenlehauer}, {\em Computing loci of rank defects of linear matrices using 
Gr\"obner bases and applications to cryptology,} In ISSAC '10: Proceedings of the 2010 international symposium on 
Symbolic and algebraic computation, ISSAC '10, pp. 257--264, New York, NY, USA, 2010.









\bibitem{GO} {\sc I. Gohberg and V. Olshevsky}, {\em Circulants,
discplacements and decompositions of matrices}, Integr. Equat. Oper. Th.,
15 (1992), pp. 730--743. 
  

\bibitem{GV} {\sc G.H. Golub and C.F. van Loan},  Matrix Computations, 
{\em The Johns Hopkins University Press,
the $3$rd ed.}, 1996.

\bibitem{GR} {\sc M.L. Green}, 
{\em Holomorphic maps into complex projective space omitting hyperplanes},
Trans. Amer. Math. Soc., 169 (1972), pp. 89--103. 



\bibitem{PHA} {\sc P. Halmos,} {\em  Bad products of good matrices,}  
Linear and Multilinear Algebra, 29  (1991),  no. 1, pp. 1--20.

\bibitem{HA} 
\textsc{J. Harris}, 
Algebraic Geometry, A First Course, 
{\em Springer, New York,} 1992.

\bibitem{HART} 
{\sc R. Hartshorne,} 
Algebraic Geometry, 
{\em Springer-Verlag, New York, Heidelberg, Berlin}, 1977.



\bibitem{HMS}
{\sc O. Holtz, V. Mehrmann and H. Schneider,} {\em Potter, Wielandt, and 
Drazin on the matrix equation $AB=\omega BA$: new answers to old 
questions,}  Amer. Math. Monthly,  111  (2004),  no. 8, pp. 655--667. 


\bibitem{HJ} {\sc R.A. Horn and C.R. Johnson},
Topics in Matrix Analysis,
{\em Cambridge Univ. Press, Cambridge,} 1991.


\bibitem{HR} {\sc M. Huhtanen}, {\em Factoring matrices into the product of two
matrices}, BIT, 47 (2007), pp. 793--808.

\bibitem{HR1} {\sc M. Huhtanen}, {\em Matrix
subspaces and determinantal hypersurfaces}, Ark. Mat.,
48 (2010), pp. 57--77.

\bibitem{HR2} {\sc M. Huhtanen}, {\em Differential
geometry of matrix inversion}, Math. Scand.,
107 (2010), pp. 267--284.





\bibitem{KO} {\sc J. Koll\'ar}, {\em Sharp effective Nullstellensatz},
J. Amer. Math. Soc., 1 (1988), pp. 963--975.

\bibitem{LA} {\sc T.J. Laffey,} {\em Conjugacy and factorization
results on matrix groups,} Polish Academy of Sciences, Warsaw (eds). 
Functional Analysis and Operator Theory, 
Banach Center Publications Number 30, (1994), pp. 203--221.

\bibitem{LAX} {\sc P. Lax}, {\em On the factorization of matrix valued
functions}, Comm. Pure Appl. Math., 24 (1976), pp. 683--688.


\bibitem{lee} {\sc J.M. Lee},  Riemannian Manifolds, An Introduction to Curvature, 
{\em Springer Verlag, New York,} 1997. 

\bibitem{leetm} {\sc J.M. Lee},   Introduction to Topological Manifolds, 
{\em Springer Verlag, New York,} 2000. 



\bibitem{Milne} {\sc J.S. Milne}, Algebraic Geometry, {\em Lecture notes at www.jmilne.org$/$math$/$}.





\bibitem{VPA}
{\sc V. Paulsen,} Completely Bounded Maps and Operator Algebras, {\em Cambridge University
Press, Cambridge,} 2002.


\bibitem{POS} {\sc M. Postnikov}, Smooth manifolds (Lectures in geometry),
{\em Mir Publishers, Moscow}, 1989. 






\bibitem{SO} {\sc M. Sombra}, {\em A sparse effective 
Nullstellensatz,} Adv. Appl. Math., 22 (1999), pp. 271--295. 

\bibitem{TA} {\sc O. Taussky,} 
{\em The characteristic polynomial and the characteristic curve 
of a pencil of matrices with complex entries,}  
\"Osterreich. Akad. Wiss. Math.-Natur. Kl. Sitzungsber., II  195  (1986),  
no. 1-3, pp. 175--178. 



\bibitem{WU} {\sc P.Y. Wu,} {\em The operator factorization problems,}
Linear Algebra Appl.,  117 (1989), pp. 35--63.

\end{thebibliography}
\end{document}